\documentclass[12pt]{amsart}

\usepackage{geometry}
\usepackage{hyperref}

\renewcommand{\hat}{\widehat}
\renewcommand{\phi}{\varphi}

\newcommand{\spec}{\operatorname{Spec}}
\newcommand{\aut}{\operatorname{Aut}}

\newcommand{\disc}{\operatorname{Disc}}
\newcommand{\gal}{\operatorname{Gal}}

\newcommand{\kgal}{K^\text{gal}}

\newcommand{\af}{\mathbb{A}}
\newcommand{\zz}{\mathbb{Z}}
\newcommand{\qq}{\mathbb{Q}}
\newcommand{\cc}{\mathbb{C}}

\newcommand{\oo}{\mathcal{O}}

\newtheorem{thm}{Theorem}[section]
\newtheorem{lm}[thm]{Lemma}

\theoremstyle{definition}
\newtheorem{defi}[thm]{Definition}

\theoremstyle{remark}
\newtheorem{rem}[thm]{Remark}

\begin{document}

\title{Upper Bounds for the Number of Number Fields with Alternating Galois Group}
\thanks{The authors are grateful for the support of the NSF in funding the Emory 2011 REU. The authors would like to thank our advisor Andy Yang, as well as Ken Ono for their guidance, useful conversations, improving the quality of exposition of this article, and hosting the REU}

\author{Eric Larson}
\address{Department of Mathematics. Harvard University, Cambridge, MA 02138.}
\email{elarson3@gmail.com}

\author{Larry Rolen}
\address{Department of Mathematics and Computer Science, Emory University, Atlanta, GA 30322.}
\email{larry.rolen@mathcs.emory.edu}

\begin{abstract}
We study the number $N(n, A_n, X)$ of number fields of degree $n$ whose
Galois closure has Galois group
$A_n$ and whose discriminant is bounded by $X$.
By a conjecture of Malle, we
expect that $N(n, A_n, X)\sim C_n\cdot X^{\frac{1}{2}} \cdot (\log X)^{b_n}$
for constants $b_n$ and $C_n$.
For $6 \leq n \leq 84393$,
the best known upper bound
is $N(n, A_n, X) \ll X^{\frac{n + 2}{4}}$;
this bound follows from Schmidt's Theorem,
which implies there are $\ll X^{\frac{n + 2}{4}}$
number fields of degree $n$.
(For $n > 84393$, there are better bounds due to
Ellenberg and Venkatesh.)
We show, using the important work of Pila on counting integral points on curves, that $N(n, A_n, X) \ll X^{\frac{n^2 - 2}{4(n - 1)}+\epsilon}$,
thereby improving the best previous exponent by approximately $\frac{1}{4}$ for $6 \leq n \leq 84393$.
\end{abstract}
\maketitle

\section{Introduction and Statement of Results}

For any positive integers $n$ and $X$ and for any fixed transitive permutation group $G$, we would like to count $N(n,G,X)$, defined to be the number of degree $n$ number fields $K$ whose Galois closure has Galois group $G$ and for which $|D_K|\leq X$. Further let $N(n,X)$ denote the number of all degree $n$ number fields with discriminant bounded in absolute value by $X$. It is an old conjecture, sometimes attributed to Linnik, that \[N(n,X)\sim c_nX \quad(n\text{ fixed, }X\rightarrow\infty).\]
The conjecture is trivial when $n=2$, and was proven for $n=3$ by Davenport and Heilbronn \cite{Dav-Heil} and for $n=4,5$ by Bhargava \cite{Bhargava-Quart}, \cite{Bhargava-Quint}. For all but finitely many $n$, the current best upper bound, due to Ellenberg and Venkatesh \cite{Ellenberg-Venk} states:
\[N(n,X)\ll (X\cdot B_n)^{\exp(C\log\sqrt{n})}.\]
Here, $B_n$ depends only on $n$ and $C$ is an absolute constant. For $6 \leq n \leq 84393$, the best bound, due to Schmidt \cite{Schmidt}, is 
\[N(n,X)\ll X^{\frac{n+2}{4}}.\]

In this note, we study the case when $G=A_n$. By a conjecture of Malle \cite{Malle}, we expect that
\[N(n,A_n,X)\buildrel{?}\over{\sim} c(n)\cdot X^{\frac{1}{2}}\cdot (\log X)^{b(n)-1},\] for some constant $c(n)$ and an explicit constant $b(n)$. Here we improve Schmidt's general bound in our case. In particular, we can use Pila's results on counting integral points on geometrically irreducible curves \cite{Pila} to show the following:

\begin{thm} \label{mainthm} We have
\[N(n, A_n, X) \ll X^{\frac{n^2 - 2}{4(n - 1)}}\cdot\log(X)^{2n+1},\]
where the implied constant depends only on $n$.
\end{thm}
\begin{rem}
Throughout this note, we write $f \ll g$ to mean that
$f \leq c \cdot g$ for a constant $c$ depending only
on the degree of number field (or degree of the algebraic
variety) in question.
\end{rem}
\indent
Note that the exponent improves on the previous record by a power of $X$ of about $\frac{1}{4}$ for these $n$. The method uses point counting on varieties in a similar manner as in \cite{Ellenberg-Venk}. The improvement follows from viewing these varieties as fibrations of curves, controlling the fibers which are not geometrically irreducible, and using a bound of Pila on counting integral points on geometrically irreducible curves.
\section{Upper Bounds via Point Counting \label{sec:varieties}}

If $K$ is a number field  of discriminant $D_K$ and degree $n$,
then Minkowski theory implies there is an element
$\alpha \in \oo_K$ of trace zero with
\[|\alpha| \ll D_K^{\frac{1}{2(n - 1)}} \quad \text{(under any archimedian valuation),}\]
where the implied constant depends only on $n$.

When $\gal(\kgal / \qq) \simeq A_n$, then $K$ must be
a primitive extension of $\qq$,
so $K = \qq(\alpha)$ and the characteristic polynomial
of $\alpha$ will determine $K$.
One can use this to give an upper bound on $N(n, A_n, X)$.
To see this, note that every pair $(K, \alpha)$ as above gives a $\zz$-point of
$\spec R$, for
\[R = \zz[x_1, x_2, \ldots, x_n]^{A_n} / (s_1) \quad \text{where} \quad s_1 = x_1 + x_2 + \cdots + x_n.\]
(Here $\zz[x_1,x_2,\ldots,x_n]^{A_n}$
denotes the ring of $A_n$-invariants in $\zz[x_1,x_2,\ldots,x_n]$.)
Now, it is a classical theorem that the ring of $A_n$-invariant functions
is generated by the symmetric functions and the square root
of the discriminant, i.e.\ we have
\[\zz[x_1, x_2, \ldots, x_n]^{A_n} \simeq \zz[s_1, s_2, \ldots, s_n, D] / \big(D^2 = \disc(t^n - s_1 t^{n - 1} + \cdots \pm s_n)\big),\]
so therefore
\[R \simeq \zz[s_1, s_2, \ldots, s_n, D] / \big(D^2 = \disc(t^n + s_2 t^{n - 2} + \cdots \pm s_n)\big).\]

Thus, to give an upper bound on $N(n, A_n, X)$, it suffices
to bound the number of $\zz$-points of $\spec R$ which satisfy the inequalities
\begin{equation} \label{bounds}
|s_j| \ll X^{\frac{j}{2(n - 1)}} \quad \text{and} \quad |D| \ll X^{\frac{n}{4}}.
\end{equation}

\section{Proof of Theorem~\ref{mainthm} when $n$ is Even \label{sec:even}}

When $n$ is even, Theorem~\ref{mainthm} is
relatively straight-forward; therefore, we begin by
examining this case.

\begin{lm}\label{evencase} Theorem~\ref{mainthm} holds when $n$ is even.
\end{lm}

\begin{proof}
By fixing $s_2, s_3, \ldots, s_{n - 1}$, we can view $\spec R$ as
a fibration of plane curves over $\af^{n - 2}$.
Each of these curves is then the zero locus of a polynomial
of the form
\begin{equation} \label{curveeven}
D^2 = \text{a polynomial of odd degree in $s_n$}.
\end{equation}

In particular, these curves are geometrically irreducible.
Therefore, we can apply Pila's bound \cite{Pila},
which states that the number of integral points on a
geometrically irreducible plane curve of degree $d$ whose coordinates
are bounded in absolute value by $B$ is at most
\[(3d)^{4d + 8} \cdot B^{\frac{1}{d}} \cdot (\log B)^{2d + 3} \ll B^{\frac{1}{d}} \cdot (\log B)^{2d + 3}.\]
For the curves defined by \eqref{curveeven}, we seek to count
integral points with
\[|s_n| \ll X^{\frac{n}{2(n - 1)}} \quad \text{and} |D| \ll X^{\frac{n}{4}}.\]
By Pila's result above, the number of such points is
\[\ll \left(X^{\frac{n}{4}}\right)^{\frac{1}{n - 1}} \cdot \left(\log (X^{\frac{n}{4}})\right)^{2 \cdot (n - 1) + 3} \ll X^{\frac{n}{4(n - 1)}} \cdot (\log X)^{(2n + 1)}.\]
Therefore, using the bounds \eqref{bounds} on the other $s_j$
from the previous section, we have
\begin{align*}
N(n, A_n, X) &\ll \left(\prod_{j = 2}^{n - 1} X^{\frac{j}{2(n - 1)}}\right) \cdot X^{\frac{n}{4(n - 1)}} \cdot (\log X)^{2n + 1} = X^{\frac{n^2 - 2}{4(n - 1)}} \cdot (\log X)^{2n + 1}. && \qedhere
\end{align*}
\end{proof}

\section{Proof of Theorem~\ref{mainthm} when $n$ is Odd}

In the case when $n$ is odd, the argument
of Lemma~\ref{evencase} breaks down because the
curves in the fibration do not have to be geometrically
irreducible. In order to circumvent this difficulty, we
will show in this section that ``most'' of the fibers
of the map $\spec R \to \af^{n - 2}$ are geometrically
irreducible.  We will then bound the number of integral points
on the fibers that fail to be geometrically irreducible.

\begin{defi} We say two polynomials $f, g \in \cc[z]$
are \emph{equivalent} if $f(z) = g(az + b)$ for some $a \in \cc^\times$
and $b \in \cc$.
\end{defi}

\begin{defi} We say that $c$ is a \emph{critical value}
of a polynomial $f$ if $c = f(d)$ for some $d$ with $f'(d) = 0$.
\end{defi}

\begin{lm} \label{fincrit}
Fix a finite set of points $S \subset \cc$
and an integer $d$. Then there are finitely many equivalence
classes of polynomials of degree $d$
whose set of critical values is contained in $S$.
\end{lm}

\begin{proof} Write $S = \{z_1, z_2, \ldots, z_n\}$,
and fix some $z_0 \notin S$. Then any polynomial of degree
$d$ whose set of critical values $f$ is contained in $S$
gives rise to a map
\[f^* \colon \pi_1(\cc - S) \to \aut(f^{-1}(z_0)) \simeq S_d.\]
Since $S$ is finite, $\pi_1(\cc - S)$ is finitely generated;
moreover, $S_d$ is finite, so
there are only finitely many possibilities for $f^*$.

Thus, it suffices to show that
any two polynomials $f$ and $g$ for which $f^* = g^*$
are equivalent. But the classical theory of covering spaces
implies that when $f^* = g^*$, then $f$ and $g$ must differ by
a deck transformation, which must be analytic because
$f$ and $g$ are analytic coverings.
The desired conclusion then follows from the well-known
fact that any automorphism
of $\hat{\cc}$ fixing $\infty$ is of the form
$z \mapsto a \cdot z + b$ with $a \in \cc^\times$
and $b \in \cc$.
\end{proof}

\begin{lm} \label{findisc} Let $n$ be an integer.
For any monic polynomial $p(z) \in \cc[z]$ of degree $n - 1$,
there are only finitely
many values of $(a_2, a_3, \cdots, a_{n - 1}) \in \cc^{n - 2}$
such that $p(z)$ is the discriminant of the polynomial
\[q(t) = t^n + a_2 t^{n - 2} + \cdots + a_{n - 1} t - z.\]
\end{lm}
\begin{proof}
In order for $p(z)$ to be the discriminant of $q(t)$,
every root $r$ of $p(z)$ must be (with multiplicity)
a critical value of the polynomial
$q_0(t) = t^n + a_1 t^{n - 1} + \cdots + a_{n - 1} t$.
Since $q_0$ is a polynomial of degree $n$, it has $n - 1$
critical values (counted with multiplicity); since $p(z)$ is a
polynomial of degree of $n - 1$, it has $n - 1$ zeros
(counted with multiplicity). Therefore, every critical
value of $q_0$ is a root of $p(z)$.
This completes the proof by Lemma~\ref{fincrit}.
\end{proof}

\begin{lm} \label{lmdim}
The locus of $(s_2, s_3, \ldots, s_{n - 1}) \in \af^{n - 2}$
such that the plane curve
\[x^2 = \disc(t^n + s_2 t^{n - 2} + \cdots \pm s_{n - 1} t - y)\]
fails to be geometrically irreducible is
an affine variety of dimension at most $\frac{n - 1}{2}$.
\end{lm}

\begin{proof} 
The corresponding plane curve fails to be geometrically
irreducible if and only if the polynomial
\[p(y) = \disc(t^n + s_2 t^{n - 2} + \cdots \pm s_{n - 1} t - y)\]
is a perfect square. But the coefficients of $p(y)$
are regular functions in $s_2, s_3, \ldots, s_{n - 1}$.
Moreover, the map $\af^{n - 2} \to \af^{n - 1}$ induced
by these regular functions is a finite map by Lemma~\ref{findisc}.

Since the locus of $(b_1, b_2, \ldots, b_{n - 1}) \in \af^{n - 1}$
such that $t^{n - 1} + b_1 t^{n - 2} + \cdots + b_{n - 1}$
is a perfect square is a Zariski-closed set of dimension
$\frac{n - 1}{2}$, this completes the proof.
\end{proof}

Using the above lemma together with the ideas
from Section~\ref{sec:even}, we can complete the proof
of Theorem~\ref{mainthm}.

\begin{proof}[Proof of Theorem~\ref{mainthm} for $n$ odd]
When $n = 3$, this follows from a result of Wright \cite{Wright},
and when $n = 5$, this follows from a result of Bhargava \cite{Bhargava-Quint}.
Thus, we can assume $n \geq 7$.

Again, we consider the fibration $\spec R \to \af^{n - 2}$
given by fixing $s_2, s_3, \ldots, s_{n - 1}$.
The argument given in Lemma~\ref{evencase} implies
the number of integral points lying on the geometrically
irreducible fibers satisfies the required bound;
it remains to see that the number of integral
points lying on the geometrically reducible fibers
also satisfies the required bound.

To prove this, we first note that by Lemma~\ref{lmdim},
all such points are contained in a subvariety of $\spec R$
of dimension at most $\frac{n - 1}{2} + 1 = \frac{n + 1}{2}$.
Moreover, the projection map $\spec R \to \af^{n - 1}$
given by fixing $s_2, s_3, \ldots, s_n$ is finite,
so it suffices to bound the number of integral points
in the box
\[|s_j| \leq X^{\frac{j}{2(n - 1)}}\]
lying in a particular affine variety
of dimension $\frac{n + 1}{2}$. But the number of such
points can be bounded by the product of the $\frac{n + 1}{2}$
largest sides of the box, and therefore is
\[\ll \prod_{j = \frac{n + 1}{2}}^n X^{\frac{j}{2(n - 1)}} = X^{\frac{(3n + 1)(n + 1)}{16(n - 1)}},\]
and therefore satisfies the required bound,
as long as $n \geq 7$.
\end{proof}

\end{document}